\documentclass{monsky2009} %remember to reinstate baselinestretch commands (there are 2)
\usepackage{amssymb}

\usepackage{bm}

\newenvironment{conjecture*}[1][]{\textbf{Conjecture #1\hspace{.3em}}}{}
\newenvironment{theorem*}[1]{\textbf{#1}\itshape \hspace{.3em}}{\upshape}
\newenvironment{remark*}[1]{\textbf{#1}\itshape \hspace{.3em}}{\upshape}
\newenvironment{example*}[1]{\textbf{#1}\itshape \hspace{.3em}}{\upshape}
\newenvironment{observation*}[1]{\textbf{#1}\itshape \hspace{.3em}}{\upshape}
\newenvironment{proof}[1][]{\textbf{Proof #1\hspace{.3em}}}{}

\newtheorem{definition}{Definition}[section]
\newtheorem{theorem}[definition]{Theorem}
\newtheorem{lemma}[definition]{Lemma}

\newcounter{kpremark}
\newcommand{\mod}[1]{\ensuremath{\hspace{.5em}(#1)}}

\newcommand{\newo}{\ensuremath{\mathcal{O}}}

\newcommand{\modd}{\ensuremath{M(\mathit{odd})}}

\newcommand{\pr}{\ensuremath{\mathit{pr}}}
\newcommand{\tr}{\ensuremath{\mathit{Tr}}}

\begin{document}

\addtolength{\parskip}{-2pt}
% The Comments accompanying this v2 should read:                        
% References added to the applications of Lemma 5.5 to Hecke algebras.

\begin{frontmatter}



%\title{}


%\author{}

\title{Generators and relations for the shallow mod~$\bm 2$ Hecke algebra in levels $\bm{\Gamma_{0}(3)}$ and $\bm{\Gamma_{0}(5)}$}
\author{Paul Monsky}

\address{Brandeis University, Waltham MA  02454-9110, USA. monsky@brandeis.edu}

\begin{abstract}
Let $\modd\subset Z/2[[x]]$ be the space of odd mod~2 modular forms of level $\Gamma_{0}(3)$. It is known that the formal Hecke operators $T_{p}:Z/2[[x]]\rightarrow Z/2[[x]]$, $p$ an odd prime other than $3$, stabilize $\modd$ and act locally nilpotently on it. So $\modd$ is an $\newo =  Z/2[[t_{5},t_{7},t_{11},t_{13}]]$-module with $t_{p}$ acting by $T_{p}$, $p\in \{5,7,11,13\}$. We show:

\begin{enumerate}
\item[(1)] Each $T_{p}:\modd\rightarrow\modd$, $p\ne 3$, is multiplication by some $u$ in the maximal ideal, $m$, of $\newo$.
\item[(2)] The kernel, $I$, of the action of $\newo$ on $\modd$ is $(A^{2},AC,BC)$ where $A,B,C$ have leading forms $t_{5}+t_{7}+t_{13},t_{7},t_{11}$.
\end{enumerate}

We prove analogous results in level $\Gamma_{0}(5)$. Now $\newo$ is $Z/2[[t_{3},t_{7},t_{11},t_{13}]]$, and the leading forms of $A,B,C$ are $t_{3}+t_{7}+t_{11},t_{7},t_{13}$.

Let $\mathit{HE}$, ``the shallow mod~2 Hecke algebra (of level $\Gamma_{0}(3)$ or $\Gamma_{0}(5)$)'' be $\newo/I$. (1) and (2) above show that $\mathit{HE}$ is a 1 variable power series ring over the 1-dimensional local ring $Z/2[[A,B,C]]/(A^{2},AC,BC)$. For another approach to all these results, based on deformation theory, see Deo and Medvedovsky \cite{4}.

%\vspace{2ex}
\end{abstract}

%\begin{keyword}

%\PACS 
%\end{keyword}
\end{frontmatter}

% main text

\section{Introduction}
\label{section1}

For $p$ an odd prime, $T_{p}: Z/2[[x]]\rightarrow  Z/2[[x]]$ is the formal Hecke operator $\sum c_{n}x^{n}\rightarrow \sum c_{pn}x^{n} + \sum c_{n}x^{pn}$; the $T_{p}$ commute. We'll be concerned with certain subspaces of $Z/2[[x]]$, coming from modular forms of level $\Gamma_{0}(N)$, and stabilized by the $T_{p}$, $p\nmid N$.  On the spaces we're looking at, each $T_{p}$ acts locally nilpotently. Let $S$ be a finite set of odd primes $p$ not dividing $N$, and $\newo$ be a power series ring over $Z/2$ in variables $t_{p}$, $p\in S$. Then our subspace is an $\newo$-module with $t_{p}$ acting by $T_{p}$. We'll show in some cases that each $T_{p}$ acting on the subspace is multiplication by an element of $\newo$ (which lies in the maximal ideal since $T_{p}$ acts locally nilpotently). And we'll describe the kernel, $I$, of the action of $\newo$ on the subspace.

The motivating level 1 example appears in \cite{3}. Let $F$ in $Z/2[[x]]$ be $x+x^{9}+x^{25}+\cdots $, the exponents being the odd squares. The subspace $V$ is spanned by the $F^{k}$, $k$ odd (and positive). $F$ is the mod~2 reduction of the weight 12 cusp form $\Delta$, and a modular forms interpretation of $V$ shows that the $T_{p}$ stabilize it; with more work one may show that they act locally nilpotently. Take $S=\{3,5\}$. Nicolas and Serre show:

\begin{enumerate}
\addtolength{\itemsep}{2pt}
\item[(1)] Each $T_{p}:V\rightarrow V$ is multiplication by an element of $Z/2[[t_{3},t_{5}]]$.
\item[(2)] $\newo = Z/2[[t_{3},t_{5}]]$ acts faithfully on $V$.
\end{enumerate}

Here is a level $\Gamma_{0}(3)$ example whose study was begun in \cite{1}.  Let $G=F(x^{3})$, and $\modd$ be spanned by the $F^{i}G^{j}$, where $i,j$ are $\ge 0$ and $i+j$ is odd. Here's a modular forms interpretation; $\modd$ consists of all odd power series that are mod~2 reductions of elements of $Z[[x]]$ arising as expansions at infinity of holomorphic modular forms of level $\Gamma_{0}(3)$ (and any weight). (We write $x$ in place of the customary $q$ for the expansion variable throughout.)  This interpretation shows that the $T_{p}$, $p\ne 3$, stabilize $\modd$. Using the local nilpotence of the $T_{p}$ acting on $V$, and on a certain subquotient $W$ of $\modd$ introduced in \cite{1}, we show that the $T_{p}$, $p\ne 3$, act locally nilpotently on $\modd$.

If we take $G$ to be $F(x^{5})$ instead of $F(x^{3})$ we get another subspace, which we'll also call $\modd$; it has a similar interpretation with $\Gamma_{0}(3)$ replaced by $\Gamma_{0}(5)$. This $\modd$ is stabilized by the $T_{p}$, $p\ne 5$, and we'll use results from \cite{2} to show that they act locally nilpotently on it.

We take $S$ to be $\{5,7,11,13\}$ in the level 3 example and to be $\{3,7,11,13\}$ in level 5.  We will show that the $T_{p}$, $p\ne 3$, are multiplication by elements of $\newo$ in the first case, while the $T_{p}$, $p\ne 5$, are multiplication by elements of $\newo$ in the second. In each case we'll determine the kernel, $I$, of the action. It is an ideal $(A^{2},AC,BC)$ where the degree 1 parts of $A$, $B$ and $C$ are linearly independent in the 4-dimensional vector space $m/m^{2}$. Apart from results from \cite{1}, \cite{2}, \cite{3} there are 2 simple new ideas. One is making use of the fact that a certain $\newo$-submodule of $V$ imbeds in the subquotient, $W$, of $\modd$. The other is showing that there are no non-zero $\newo$-linear maps $W\rightarrow V$.

Shaunak Deo and Anna Medvedovsky, \cite{4}, have derived the same results simultaneously with us.  They use techniques from deformation theory in place of our arguments, which are more related to ideas from \cite{3}.  Communications in both directions were helpful in understanding precisely what the kernel, $I$, of the action of $\newo$ on $\modd$ should be, and in completing the proofs, both for us in our arguments and for them in theirs. It would be interesting to understand how the proofs are related.

\section{$\bm\modd $ in level $\bm\Gamma_{0}(3)$}
\label{section2}

Throughout this section $\pr : Z/2[[x]]\rightarrow Z/2[[x]]$ is the map $\sum c_{n}x^{n}\rightarrow \sum_{(n,3)=1} c_{n}x^{n}$, $G$ is $F(x^{3})$ and $D=\pr(F)$. (We'll use a related but different notation in the next section.)

\begin{definition}
\label{def2.1}
$U_{3} : Z/2[[x]]\rightarrow Z/2[[x]]$ is the map $\sum c_{n}x^{n}\rightarrow \sum c_{3n}x^{n}$. $\modd\subset Z/2[[x]]$ is spanned by the $F^{i}G^{j}$, $i,j\ge 0$, $i+j$ odd.
\end{definition}

As was shown in \cite{1}, there is an interpretation of $\modd$ in terms of modular forms of level $\Gamma_{0}(3)$ that shows that the $T_{p}$, $p\ne 3$, stabilize it. It's also stabilized by $U_{3}$ (by a similar argument) but we'll only need the obvious fact that $U_{3}$ maps the space spanned by the $G^{k}$, $k$ odd, bijectively to $V$, and that this map commutes with $T_{p}$, $p\ne 3$. The following are proved in \cite{1}:

\begin{theorem}%[]
\label{theorem2.2}\hspace{2em}
\begin{enumerate}
\item[(1)] $F$ has degree 4 over $Z/2(G)$ and $(F+G)^{4}=FG$.
\item[(2)] As $Z/2[G^{2}]$-module, $\modd$ has basis $\{G,F,F^{2}G,F^{3}\}$.
\item[(3)] The trace map $Z/2(F,G)\rightarrow Z/2(G)$ takes $G,F,F^{2}G,F^{3}$ to $0,0,0,G$.  So it gives a $Z/2[G^{2}]$-linear map $\tr : \modd\rightarrow\modd$. The kernel $N2$ and image $N1$ of $\tr$ have $Z/2[G^{2}]$-bases $\{G,F,F^{2}G\}$ and $\{G\}$.
\item[(4)] The filtration $\modd\supset N2\supset N1 \supset (0)$ of $\modd$ is ``Hecke-stable.''  I.e., the $T_{p}$, $p\ne 3$, stabilize $N2$ and $N1$.
\item[(5)] The image, $W$, of $N2$ under $\pr$ has $Z/2[G^{2}]$-basis $\{D,D^{2}G\}$. $N1$ is the kernel of $\pr : N2\rightarrow W$, and so the $T_{p}$, $p\ne 3$, stabilize $W$, and the isomorphism $N2/N1\rightarrow W$ commutes with $T_{p}$, $p\ne 3$.
\end{enumerate}
\end{theorem}

\begin{remark*}{Remarks}
\label{remarkafter2.2}
Besides the above bijection $N2/N1\rightarrow W$ commuting with $T_{p}$, $p\ne 3$, we have the bijection $U_{3}:N1\rightarrow V$ commuting with these $T_{p}$. Finally there is a composite bijection $\modd/N2\stackrel{\tr}{\rightarrow} N1\stackrel{U_{3}}{\rightarrow} V$. We'll show that this too commutes with $T_{p}$.
\end{remark*}

\begin{lemma}%[]
\label{lemma2.3}
$T_{3}:V\rightarrow V$ is onto.
\end{lemma}

\begin{proof}
By \cite{3}, $V$ has a $Z/2$-basis $\{m_{i,j}\}$ with $m_{0,0}=F$, ``adapted to $T_{3}$ and $T_{5}$.''  But then $T_{3}$ takes $m_{i+1,j}$ to $m_{i,j}$.
\qed
\end{proof}

\begin{observation*}{Observation}
\label{observation1}
Here's an ``injective modules are divisible'' generalization of the above. Suppose $M$ is a $k[[X,Y]]$-module that admits a $k$-basis $m_{i,j}$ adapted to $X$ and $Y$. Then if $u$ in $k[[X,Y]]$ is non-zero, $uM=M$. (And in particular, the action is faithful.)  To see this, totally order $N\times N$, taking $(0,0)<(1,0)<(0,1)<(2,0)<(1,1)<(0,2)<(3,0)<\cdots $. Let $cX^{a}Y^{b}$ be the monomial appearing in the leading form of $u$ with largest $a$. Then $u(m_{a+i,b+j})=cm_{i,j} +$ a $k$-linear combination of $m_{r,s}$ with $(r,s)<(i,j)$, and an inductive argument using the total ordering gives the result.
\end{observation*}

\begin{lemma}%[]
\label{lemma2.4}
The composite map $V\stackrel{\tr}{\rightarrow} N1\stackrel{U_{3}}{\rightarrow} V$ is $T_{3}$, and so, by Lemma \ref{lemma2.3}, is onto.
\end{lemma}

\begin{proof}
Let $U(X,Y)$ be $(X+Y)^{4}+XY$. Then $U(F(x),F(x^{3}))=0$. Replacing $x$ by $x^{3}$ we find that $U(G,F(x^{9}))=0$. So $U(F(x^{9}),G)=0$, and $F(x^{9})$ is a conjugate of $F$ over $Z/2(G)$. Similarly, replacing $x$ by $rx$ where $r$ is in the field of 4 elements, $r^{3}=1$, we find that the other 3 conjugates of $F$ are the $F(rx)$. So the conjugates of $F^{k}$ are $F^{k}(x^{9})$ and $F^{k}(rx)$. Writing $F^{k}$ as $\sum c_{n}x^{n}$, adding together the conjugates, and applying $U_{3}$ we get $T_{3}(F^{k})$.
\qed
\end{proof}

\begin{theorem}%[]
\label{theorem2.5}
The composite bijection $\modd/N2\stackrel{\tr}{\rightarrow} N1\stackrel{U_{3}}{\rightarrow} V$ commutes with $T_{p}$, $p\ne 3$. We conclude that $\modd/N2$, $N2/N1$ and $N1$ identify with $V$, $W$ and $V$ as Hecke-modules.
\end{theorem}

\begin{proof}
By Lemma \ref{lemma2.4} the restriction of our map to $V$ is onto So the elements of $V$ span $\modd/N2$. And on $V$ our map is $T_{3}:V\rightarrow V$ which commutes with the $T_{p}:V\rightarrow V$.
\qed
\end{proof}

\begin{theorem}%[]
\label{theorem2.6}
The $T_{p}$, $p\ne 3$, act locally nilpotently on $\modd$. In other words, if $f\in \modd$ and $p\ne 3$, some power of $T_{p}$ annihilates $f$.
\end{theorem}

\begin{proof}
\cite{3} and \cite{1} show that $T_{p}$ acts locally nilpotently on $V$ and $W$. And the quotients in the filtration of $\modd$ are Hecke-isomorphic to $V$, $W$ and $V$.
\qed
\end{proof}

For the rest of this section, unless otherwise noted, $S=\{5,7,11,13\}$ and $\newo$ is the 4-variable power series ring over $Z/2$ in the $t_{p}$, $p\in S$. Then $V$, $W$ and $\modd$ are all $\newo$-modules with $t_{p}$ acting by $T_{p}$, $p\in S$. Let $I(V)$, $I(W)$ and $I$ be the kernels of the respective actions.

$I(V)$ is easily described. As we noted in section \ref{section1}, when $V$ is viewed as a $Z/2[[t_{3},t_{5}]]$-module, the action is faithful, and each $T_{p}:V\rightarrow V$ is multiplication by some $u$ in $Z/2[[t_{3},t_{5}]]$. In \cite{3} it's shown that: 
\vspace{1ex}

\centerline{
\begin{tabular}{lll}
when & $p=11$ & $u=\mathrm{unit} (t_{3})$\\
when & $p=13$ & $u=\mathrm{unit} (t_{5})$\\
when & $p=7$ & $u=\mathrm{unit} (t_{3})(t_{5})$\\
\end{tabular}
}

It follows from the above that when $V$ is viewed as a $Z/2[[t_{11},t_{13}]]$-module the action is faithful, and each $T_{p}:V\rightarrow V$ is multiplication by some $u$ in $Z/2[[t_{11},t_{13}]]$. Furthermore when $p=5$, $u=\mathrm{unit} (t_{13})$, while when $p=7$, $u=\mathrm{unit} (t_{11})(t_{13})$. So $I(V)$ is generated by 2 elements, $t_{5}+\mathrm{unit} (t_{13})$ and $t_{7}+\mathrm{unit} (t_{11})(t_{13})$.  This gives:

\begin{theorem}%[]
\label{theorem2.7}
$I(V)$ is generated by 2 elements $A$ and $B$ whose leading forms can be taken to be $t_{5}+t_{7}+t_{13}$ and $t_{7}$.
\end{theorem}

To describe $I(W)$ we use the following results from \cite{1}:

\begin{theorem}%[]
\label{theorem2.8}
Let $W1$ and $W5$ be the $Z/2[G^{2}]$-submodules of $W$ generated by $D$ and $D^{2}G$ respectively. (In fact $G=D^{3}$, so that a $Z/2$-basis of $W1$ consists of the $D^{k}$, $k\equiv 1\mod{6}$, while a $Z/2$-basis of $W5$ consists of the $D^{k}$, $k\equiv 5\mod{6}$.)

\begin{enumerate}
\addtolength{\itemsep}{2pt}
\item[(1)] The $T_{p}$, $p\equiv 1\mod{6}$, stabilize $W1$ and $W5$. The $T_{p}$, $p\equiv 5\mod{6}$, map $W1$ to $W5$, and $W5$ to $W1$.
\item[(2)] $W1$ has a basis $\{m_{i,j}\}$ with $m_{0,0}=D$, adapted to $T_{7}$ and $T_{13}$. The same holds for $W5$ with $m_{0,0}=D^{5}=D^{2}G$. It follows that $T_{7}$ and $T_{13}$ act locally nilpotently on $W1$, $W5$ and $W$.
\item[(3)] Taking $S=\{7,13\}$ and making $W$ into a $Z/2[[t_{7},t_{13}]]$-module, we find that each $T_{p}:W\rightarrow W$, $p\equiv 1\mod{6}$, is multiplication by some $u$ in the maximal ideal of $Z/2[[t_{7},t_{13}]]$.
\item[(4)] And each $T_{p}:W\rightarrow W$, $p\equiv 5\mod{6}$, is the composition of $T_{5}$ with multiplication by some $u$ in $Z/2[[t_{7},t_{13}]]$.
\item[(5)] $T_{5}^{2}:W\rightarrow W$ is multiplication by $\lambda^{2}$ for some $\lambda$ in $Z/2[[t_{7},t_{13}]]$ with leading form $t_{7}+t_{13}$.
\end{enumerate}
\end{theorem}

%\enlargethispage{8pt}

\begin{remark*}{Remarks}
\label{remarkafter2.8}
(3), (4) and (5) show that each $T_{p}$, $p\ne 3$, acts locally nilpotently on $W$. And if we take $S=\{5,7,13\}$, each $T_{p}:W\rightarrow W$, $p\ne 3$, is multiplication by an element of $Z/2[[t_{5},t_{7},t_{13}]]$. Furthermore if we set $\varepsilon = t_{5}+\lambda$, then $\varepsilon^{2}$ kills $W$. Note that the leading form of $\varepsilon$ is $t_{5}+t_{7}+t_{13}$.
\end{remark*}

\begin{theorem}%[]
\label{theorem2.9}
$I(W)$ is generated by 2 elements $\varepsilon^{2}$ and $C$ where $\varepsilon$ is congruent to the $A$ of Theorem \ref{theorem2.7} mod~$m^{2}$, and the leading form of $C$ is $t_{11}$.
\end{theorem}

\begin{proof}
First we determine the kernel of the action of $Z/2[[t_{5},t_{7},t_{13}]]$ on $W$. The remark above shows that the kernel contains $(\varepsilon^{2})=(t_{5}^{2}+\lambda^{2})$. If $R$ is in the kernel, we may modify $R$ by an element of this ideal, and assume that $R=u+u^{\prime}t_{5}$ with $u$ and $u^{\prime}$ in $Z/2[[t_{7},t_{13}]]$. Since $u$ stabilizes $W1$ and $W5$ while $u^{\prime}t_{5}$ takes $W1$ to $W5$ and $W5$ to $W1$, $u$ and $\lambda^{2}u^{\prime}$ are elements of $Z/2[[t_{7},t_{13}]]$ annihilating $W$. Since $Z/2[[t_{7},t_{13}]]$ acts faithfully on $W$ (see for example the observation following Lemma \ref{lemma2.3}), $u=u^{\prime}=0$, and the kernel of the action is $(\varepsilon^{2})$; note that $\varepsilon$ has the same leading form as $A$. Finally by (4) of Theorem \ref{theorem2.8}, $I(W)$ contains an element of the form $t_{11}+vt_{5}$ with $v$ in $\newo$. It remains to show that $v$ is not a unit. But an easy calculation shows that $T_{11}(D^{5})=0$ while $T_{5}(D^{5})=D$. 
\qed
\end{proof}

\begin{definition}
\label{def2.10}
$V(0,\star)$ is the kernel of $T_{3}:V\rightarrow V$.
\end{definition}

If $\{m_{i,j}\}$ is a basis of $V$ adapted to $T_{3}$ and $T_{5}$, the $m_{0,j}$ form a $Z/2$-basis of $V(0,\star)$. $Z/2[[t_{5}]]$ acts faithfully and locally nilpotently on $V(0,\star)$. $V(0,\star)$ is an $\newo$-submodule of $V$, and the kernel of the action is a height 3 prime ideal $P$ generated by $t_{7}$, $t_{11}$ and an element with leading form $t_{13}+t_{5}$.

\begin{lemma}%[]
\label{lemma2.11}
$W$, as well as $V$, contains a ``Hecke-submodule'' isomorphic to $V(0,\star)$.
\end{lemma}

\begin{proof}
If $f\in V(0,\star)$, $T_{3}(f)=U_{3}(\tr(f))=0$. Since $U_{3}:N1\rightarrow V$ is bijective, $\tr(f)=0$ and $f$ is in $N2$. So we have a composite map $V(0,\star)\subset N2\stackrel{\pr}{\rightarrow}W$ which commutes with $T_{p}$, $p\ne 3$, and takes $m_{0,0}=F$ to $D$. Since $m_{0,0}$ is not in the kernel of this map, the kernel is $(0)$, giving the result.
\qed
\end{proof}

\begin{theorem}%[]
\label{theorem2.12}
Let $A,B,C$ be as in Theorems \ref{theorem2.7} and \ref{theorem2.9}. Then $(A,B,C)=I(V)+I(W)=P$, the kernel of the action of $\newo$ on $V(0,\star)$.
\end{theorem}

\begin{proof}
Evidently $(A,B,C)\subset I(V)+I(W)$. Since $V$ and $W$ each have an $\newo$-submodule isomorphic to $V(0,\star)$, $I(V)+I(W)\subset P$. Finally $(A,B,C)$ and $P$ are height 3 primes of $\newo$.
\qed
\end{proof}

\begin{theorem}%[]
\label{theorem2.13}
The only $\newo$-linear map $\alpha : W\rightarrow V$ is the zero-map.
\end{theorem}

\begin{proof}
$\alpha(W)$ is annihilated both by $I(W)$ and $I(V)$. So by Theorem \ref{theorem2.12} it is annihilated by $P$, and thus by $t_{7}$. Then $\alpha (t_{7}W)=t_{7}\alpha(W)=(0)$. But since $W1$ and $W5$ have bases adapted to $T_{7}$ and $T_{13}$, $t_{7}(W)=W$, and $\alpha(W)=(0)$.
\qed
\end{proof}

\begin{theorem}%[]
\label{theorem2.14}
If $p\ne 3$, $T_{p}:\modd\rightarrow\modd$ is multiplication by some $u$ in $\newo$.
\end{theorem}

\begin{proof}
We know that $T_{p}:V\rightarrow V$ and $W\rightarrow W$ are multiplication by some  $u$ and $u^{\prime}$ in $\newo$; for $W$ see the remarks following Theorem \ref{theorem2.8}. Then $f\rightarrow T_{p}(f)-uf$ and $f\rightarrow T_{p}(f)-u^{\prime}f$ both annihilate $V(0,\star)$ by Lemma \ref{lemma2.11}. So $u-u^{\prime}$ is in $P$, and by Theorem \ref{theorem2.12} it is in $(A,B,C)$. Modifying $u$ by an element of $(A,B)$, and $u^{\prime}$ by an element of $(C)$ we may assume $u-u^{\prime}=0$. Let $\alpha :\modd\rightarrow\modd$ be the map $f\rightarrow T_{p}(f)-uf$. We'll show that $\alpha$ is the zero-map.

$\alpha$ annihilates $V$, and since $u=u^{\prime}$, it annihilates $W$. Since $\alpha$ commutes with $U_{3}$ and $\pr$ it annihilates $N1$ and $N2/N1$. So $\alpha(N2)\subset N1$, and $\alpha$ induces an $\newo$-linear map $N2/N1\rightarrow N1$. By Theorem \ref{theorem2.13} this is the zero-map, and $\alpha(N2)=(0)$. But the proof of Theorem \ref{theorem2.5} shows that the elements of $V$ span $\modd/N2$. Since $\alpha(V)=0$, $\alpha=0$.
\qed
\end{proof}

\begin{theorem}%[]
\label{theorem2.15}
There are $A,B,C$ in $\newo$ with leading forms $t_{5}+t_{7}+t_{13}$, $t_{7}$ and $t_{11}$ such that $I(V)=(A,B)$ and $I(W)=(A^{2},C)$. The kernel, $I$, of the action of $\newo$ on $\modd$ is $I(V)\cap I(W)=(A^{2},AC,BC)$.
\end{theorem}

\begin{proof}
Let $A,B,C,\varepsilon$ be as in Theorems \ref{theorem2.7} and \ref{theorem2.9}; $I(V)=(A,B)$, $I(W)=(\varepsilon^{2},C)$ and $A-\varepsilon$ is in $m^{2}$. Then $A^{2}-\varepsilon^{2}$ is in $I(V)+I(W)$ which is $(A,B,C)$ by Theorem \ref{theorem2.12}. Since $(A,B,C)$ is prime, $A-\varepsilon$ is in $(A,B,C)\cap m^{2}=mA+mB+mC$. Modifying $A$ by an element of $mA+mB$, and $\varepsilon$ by an element of $mC$ we may assume $A-\varepsilon=0$. Then $I(V)=(A,B)$, $I(W)=(A^{2},C)$ and it follows that $I(V)\cap I(W)=(A^{2},AC,BC)$. Suppose $u$ is in $I(V)\cap I(W)$. Let $\alpha :\modd\rightarrow\modd$ be multiplication by $u$. Then $\alpha(N1)=(0)$, $\alpha(N2)\subset N1$, and arguing as in the final paragraph of the proof of Theorem \ref{theorem2.14} we get the result; $u$ is in $I$.
\qed
\end{proof}

\section{$\bm\modd $ in level $\bm\Gamma_{0}(5)$}
\label{section3}

Throughout this section $\pr : Z/2[[x]]\rightarrow Z/2[[x]]$ is the map $\sum c_{n}x^{n}\rightarrow \sum_{(n,5)=1} c_{n}x^{n}$, $G$ is $F(x^{5})$ and $D=\pr(F)$.

\begin{definition}
\label{def3.1}
$U_{5} : Z/2[[x]]\rightarrow Z/2[[x]]$ is the map $\sum c_{n}x^{n}\rightarrow \sum c_{5n}x^{n}$. $\modd\subset Z/2[[x]]$ is spanned by the $F^{i}G^{j}$, $i,j\ge 0$, $i+j$ odd.
\end{definition}

As was shown in \cite{2}, there is an interpretation of $\modd$ in terms of modular forms of level $\Gamma_{0}(5)$ that shows that the $T_{p}$, $p\ne 5$, stabilize it. It is also stabilized by $U_{5}$, but we'll only need the obvious fact that $U_{5}$ maps the space spanned by the $G^{k}$, $k$ odd, bijectively to $V$, and that this map commutes with $T_{p}$, $p\ne 5$. The following are proved in \cite{2}:

\begin{theorem}%[]
\label{theorem3.2}\hspace{2em}
\begin{enumerate}
\addtolength{\itemsep}{2pt}
\item[(1)] $F$ has degree 6 over $Z/2(G)$ and $(F+G)^{6}=FG$.
\item[(2)] As $Z/2[G^{2}]$-module, $\modd$ has basis $\{G,F,F^{2}G,F^{3},F^{4}G,F^{5}\}$.
\item[(3)] The trace map $Z/2(F,G)\rightarrow Z/2(G)$ takes $G,F,F^{2}G,F^{3},F^{4}G,F^{5}$ to $0,0,0,0,0,G$.  So it gives a $Z/2[G^{2}]$-linear map $\tr : \modd\rightarrow\modd$. The kernel $N2$ and image $N1$ of $\tr$ have $Z/2[G^{2}]$-bases $\{G,F,F^{2}G,F^{3},F^{4}G\}$ and $\{G\}$.
\item[(4)] The filtration $\modd\supset N2\supset N1 \supset (0)$ of $\modd$ is ``Hecke-stable.''  I.e., the $T_{p}$, $p\ne 5$, stabilize $N2$ and $N1$.
\item[(5)] The image, $W$, of $N2$ under $\pr$ has $Z/2[G^{2}]$-basis $\{D,D^{8}/G,D^{2}G,D^{4}G\}$. $N1$ is the kernel of $\pr : N2\rightarrow W$, and so the $T_{p}$, $p\ne 5$, stabilize $W$, and the isomorphism $N2/N1\rightarrow W$ commutes with $T_{p}$, $p\ne 5$.
\end{enumerate}
\end{theorem}

\begin{remark*}{Remarks}
\label{remarkafter3.2}
Note that $\pr$ takes the elements $F$, $F^{2}G$ and $F^{4}G$ of $N2$ to $D$, $D^{2}G$ and $D^{4}G$.  Also $F(F+G)^{2}= F(F+G)^{8}/FG= (F^{8}/G)+G^{7}$; $\pr$ takes this to $D^{8}/G$. Besides the bijection $N2/N1\rightarrow W$ commuting with $T_{p}$, $p\ne 5$, we have the bijection $U_{5}:N1\rightarrow V$ commuting with these $T_{p}$. Finally there is a composite bijection $\modd/N2\stackrel{\tr}{\rightarrow} N1\stackrel{U_{5}}{\rightarrow} V$; we'll show that this too commutes with $T_{p}$.
\end{remark*}

\begin{lemma}%[]
\label{lemma3.3}
$T_{5}:V\rightarrow V$ is onto.
\end{lemma}

\begin{proof}
If $m_{i,j}$ are as in Lemma \ref{lemma2.3}, $T_{5}$ takes $m_{i,j+1}$ to $m_{i,j}$.
\qed
\end{proof}
\pagebreak

\begin{lemma}%[]
\label{lemma3.4}
The composite map $V\stackrel{\tr}{\rightarrow} N1\stackrel{U_{5}}{\rightarrow} V$ is $T_{5}$, and so by Lemma \ref{lemma3.3}, is onto.
\end{lemma}

\begin{proof}
We argue as in Lemma \ref{lemma2.4}. Now, however, $U$ is $(X+Y)^{6}+XY$, and the conjugates of $F$ over $Z/2(G)$ are $F(x^{25})$ and the $F(rx)$ where $r^{5}=1$. The argument is otherwise unchanged.
\qed
\end{proof}

\begin{theorem}%[]
\label{theorem3.5}
The composite bijection $\modd/N2\stackrel{\tr}{\rightarrow} N1\stackrel{U_{5}}{\rightarrow} V$ commutes with $T_{p}$, $p\ne 5$. Together with the remarks following Theorem \ref{theorem3.2}, this shows that $\modd/N2$, $N2/N1$ and $N1$ identify with $V$, $W$ and $V$ as Hecke-modules.
\end{theorem}

\begin{proof}
See the proof of Theorem \ref{theorem2.5}.
\qed
\end{proof}

\begin{theorem}%[]
\label{theorem3.6}
The $T_{p}$, $p\ne 5$, act locally nilpotently on $\modd$. In other words, if $f\in \modd$ and $p\ne 5$, some power of $T_{p}$ annihilates $f$.
\end{theorem}

\begin{proof}
\cite{3} and \cite{2} show that $T_{p}$ acts locally nilpotently on $V$ and $W$. And the quotients in the filtration of $\modd$ are Hecke-isomorphic to $V$, $W$ and $V$.
\qed
\end{proof}

For the rest of this section, unless otherwise noted, $S=\{3,7,11,13\}$ and $\newo$ is the 4-variable power series ring over $Z/2$ in the $t_{p}$, $p\in S$. Then $V$, $W$ and $\modd$ are all $\newo$-modules with $t_{p}$ acting by $T_{p}$, $p\in S$. Let $I(V)$, $I(W)$ and $I$ be the kernels of the respective actions.

$I(V)$ is easily described. As in the paragraph before Theorem \ref{theorem2.7} we view $V$ as a $Z/2[[t_{11},t_{13}]]$-module.  When $p=3$, $T_{p}:V\rightarrow V$ is multiplication by $\mathrm{unit} (t_{11})$, while when $p=7$, $T_{p}$ is multiplication by $\mathrm{unit} (t_{11})(t_{13})$. So $I(V)$ is generated by two elements $t_{3}+\mathrm{unit} (t_{11})$ and $t_{7}+\mathrm{unit} (t_{11})(t_{13})$, and:

\begin{theorem}%[]
\label{theorem3.7}
$I(V)$ is generated by 2 elements $A$ and $B$ whose leading forms can be taken to be $t_{3}+t_{7}+t_{11}$ and $t_{7}$.
\end{theorem}

To describe $I(W)$ we use the following results from \cite{2}:

\begin{theorem}%[]
\label{theorem3.8}
Let $D_{1},D_{3},D_{7},D_{9}$ be $D,D^{8}/G,D^{2}G$ and $D^{4}G$. Define $D_{k}$ for $k$ positive and prime to $10$ by $D_{k+10}=G^{2}D_{k}$. Let $W_{a}$ be spanned by $D_{k}$, $k\equiv 1,3,7,9\mod{20}$, and $W_{b}$ be spanned by $D_{k}$, $k\equiv 11,13,17,19\mod{20}$. Then $W=W_{a}\oplus W_{b}$ and:
\begin{enumerate}
\addtolength{\itemsep}{2pt}
\item[(1)] The $T_{p}$, $p\equiv 1,3,7,9\mod{20}$, stabilize $W_{a}$ and $W_{b}$. The $T_{p}$, $p\equiv 11,13,17,19\mod{20}$, map $W_{a}$ to $W_{b}$ and $W_{b}$ to $W_{a}$.
\item[(2)] $W_{a}$ has a basis $\{m_{i,j}\}$ with $m_{0,0}=D$ adapted to $T_{3}$ and $T_{7}$. The same holds for $W_{b}$ with $m_{0,0}=D_{11}$. It follows that $T_{3}$ and $T_{7}$ act locally nilpotently on $W_{a}$, $W_{b}$ and $W$.
\item[(3)] Taking $S=\{3,7\}$ and making $W$ into a $Z/2[[t_{3},t_{7}]]$-module, we find that each $T_{p}:W\rightarrow W$, $p\equiv 1,3,7,9\mod{20}$, is multiplication by some $u$ in the maximal ideal of $Z/2[[t_{3},t_{7}]]$.
\item[(4)] And each $T_{p}:W\rightarrow W$, $p\equiv 11,13,17,19\mod{20}$ is the composition of $T_{11}$ with multiplication by some $u$ in $Z/2[[t_{3},t_{7}]]$.
\item[(5)] $T_{11}^{2}:W\rightarrow W$ is multiplication by $\lambda^{2}$ for some $\lambda$ in $Z/2[[t_{3},t_{7}]]$ with leading form $t_{3}+t_{7}$.
\end{enumerate}
\end{theorem}

\begin{remark*}{Remarks}
\label{remarkafter3.8}
Now each $T_{p}:W\rightarrow W$, $p\ne 5$, is locally nilpotent on $W$. And if $S=\{3,7,11\}$, each $T_{p}$ is multiplication by an element of $Z/2[[t_{3},t_{7},t_{11}]]$. And if we set $\varepsilon = t_{11}+\lambda$, then $\varepsilon^{2}$ kills $W$. Note that the leading form of $\varepsilon$ is $t_{3}+t_{7}+t_{11}$.
\end{remark*}

\begin{theorem}%[]
\label{theorem3.9}
$I(W)$ is generated by 2 elements $\varepsilon^{2}$ and $C$, where $\varepsilon$ is congruent mod~$m^{2}$ to the $A$ of Theorem \ref{theorem3.7} and the leading form of $C$ is $t_{13}$.
\end{theorem}

\begin{proof}
The proof is essentially the same as that of Theorem \ref{theorem2.9}. At the very end we use (4) of Theorem \ref{theorem3.8} to see that $I(W)$ contains an element of the form $t_{13}+vt_{11}$ with $v$ in $\newo$. But an easy calculation shows that $T_{13}(D_{11})=T_{13}(DG^{2})=0$ while $T_{11}(D_{11})=T_{11}(x^{11}+\cdots)=x+\cdots\ne 0$. So $v$ is not a unit, and our element has leading form $t_{13}$.
\qed
\end{proof}

\begin{definition}
\label{def3.10}
$V(\star,0)$ is the kernel of $T_{5}:V\rightarrow V$.
\end{definition}

If $\{m_{i,j}\}$ is a basis of $V$ adapted to $T_{3}$ and $T_{5}$, the $m_{i,0}$ form a $Z/2$-basis of $V(\star,0)$. $Z/2[[t_{3}]]$ acts faithfully and locally nilpotently on $V(\star,0)$. $V(\star,0)$ is an $\newo$-submodule of $V$, and the kernel of the action is a height 3 prime ideal, $P$, generated by $t_{7}$, $t_{13}$ and an element with leading form $t_{11}+t_{3}$.

\begin{lemma}%[]
\label{lemma3.11}
$W$, as well as $V$, contains a ``Hecke-submodule'' isomorphic to $V(\star,0)$.
\end{lemma}

\begin{proof}
If $f\in V(\star,0)$, $T_{5}(f)=U_{5}(\tr(f))=0$. As in the proof of Lemma \ref{lemma2.11} we conclude that $f$ is in $N2$ and we get a composite map $V(\star,0)\subset N2\stackrel{\pr}{\rightarrow}W$ taking $m_{0,0}$ to $D$, which is the desired imbedding.
\qed
\end{proof}

\begin{theorem}%[]
\label{theorem3.12}
Let $A,B,C$ be as in Theorems \ref{theorem3.7} and \ref{theorem3.9}. Then $(A,B,C)=I(V)+I(W)=P$, the kernel of the action of $\newo$ on $V(\star,0)$.
\end{theorem}

The proof mimics that of Theorem \ref{theorem2.12}.

\begin{theorem}%[]
\label{theorem3.13}
The only $\newo$-linear map $\alpha : W\rightarrow V$ is the zero-map.
\end{theorem}

The proof is like that of Theorem \ref{theorem2.13}, but this time we use the fact that $W_{a}$ and $W_{b}$ have bases adapted to $T_{3}$ and $T_{7}$ to see that $T_{7}(W)=W$.

\begin{theorem}%[]
\label{theorem3.14}
If $p\ne 5$, $T_{p}:\modd\rightarrow\modd$ is multiplication by some $u$ in $\newo$.
\end{theorem}

We argue as in the proof of Theorem \ref{theorem2.14}, using $V(\star,0)$ in place of $V(0,\star)$.

\begin{theorem}%[]
\label{theorem3.15}
There are $A,B,C$ in $\newo$ with leading forms $t_{3}+t_{7}+t_{11}$, $t_{7}$ and $t_{13}$ such that $I(V)=(A,B)$ and $I(W)=(A^{2},C)$. The kernel, $I$, of the action of $\newo$ on $\modd$ is $I(V)\cap I(W)=(A^{2},AC,BC)$.
\end{theorem}

The proof mimics that of Theorem \ref{theorem2.15}.

% \appendix

% \section{}
% \label{}

\end{document}